\pdfoutput=1

\documentclass{jgcc}

\usepackage{lastpage}

\jgccdoi{15}{2}{2}{12496} 

\jgccheading{}{\pageref{LastPage}}{}{}{Nov.~1,~2023}{Apr.~19,~2024}{}   

\usepackage{amsmath,amssymb}
\usepackage{algorithmic}

\newtheorem{algorithm}[thm]{Algorithm}

\newcommand\ssum{\textstyle\sum\limits}

\newcommand{\norm}[1]{\left\lVert#1\right\rVert}

\DeclareMathOperator{\LT}{LT_\sigma}
\DeclareMathOperator{\LC}{LC_\sigma}
\DeclareMathOperator{\LM}{LM_\sigma}
\DeclareMathOperator{\NF}{NF}
\DeclareMathOperator{\Spec}{Spec}
\DeclareMathOperator{\minPrimes}{minPrimes}

\DeclareMathOperator{\RV}{RV}

\DeclareMathOperator{\Rad}{Rad}
\DeclareMathOperator{\Ker}{Ker}
\DeclareMathOperator{\length}{length}
\DeclareMathOperator{\rank}{rank}
\DeclareMathOperator{\Mat}{Mat}
\DeclareMathOperator{\Syz}{Syz}
\DeclareMathOperator{\sqfree}{sqfree}
\DeclareMathOperator{\Frob}{Frob}

\newcommand\Z{\mathbb{Z}}
\newcommand\F{\mathbb{F}}
\newcommand\Q{\mathbb{Q}}

\title{Efficient Algorithms for Finite $\Z$-Algebras}

\author[M.~Kreuzer]{Martin Kreuzer}
\address{Fakult\"{a}t f\"{u}r Informatik und Mathematik \\
Universit\"{a}t Passau, D-94030 Passau, Germany}
\email{martin.kreuzer@uni-passau.de}

\author[F.~Walsh]{Florian Walsh}
\address{Fakult\"{a}t f\"{u}r Informatik und Mathematik \\
Universit\"{a}t Passau, D-94030 Passau, Germany}
\email{florian.walsh@uni-passau.de}

\begin{document}

\begin{abstract}
For a finite $\Z$-algebra $R$, i.e., for a $\Z$-algebra which is a finitely 
generated $\Z$-module, we assume that~$R$ is explicitly given by a system of 
$\Z$-module generators~$G$, its relation module $\Syz(G)$, and the structure 
constants of the multiplication in~$R$. In this setting we develop and analyze 
efficient algorithms for computing essential information about~$R$. First we provide polynomial
time algorithms for solving linear systems of equations over~$R$ and for basic
ideal-theoretic operations in~$R$. Then we develop ZPP (zero-error probabilitic polynomial time)
algorithms to compute the nilradical and the maximal ideals of 0-dimensional affine algebras
$K[x_1,\dots,x_n]/I$ with $K=\Q$ or $K=\F_p$. The task of finding the associated primes
of a finite $\Z$-algebra~$R$ is reduced to these cases and solved in ZPPIF (ZPP plus one
integer factorization). With the same complexity, we calculate the connected components 
of the set of minimal associated primes $\minPrimes(R)$ and then the primitive idempotents of~$R$.
Finally, we prove that knowing an explicit representation of~$R$ is polynomial time
equivalent to knowing a strong Gr\"obner basis of an ideal~$I$ such that 
$R = \Z[x_1,\dots,x_n]/I$.
\end{abstract}



\maketitle

%
%

\section{Introduction}
\label{sec1}

Computing the radical and the primary decomposition of an ideal, the associated primes and 
the primitive idempotents of an algebra, or the connected components of its spectrum, are 
among the hardest tasks in Computer Algebra. For a finitely generated algebra $R=K[x_1,\dots,x_n]/I$
over a field~$K$ with an ideal~$I$ that is given by its generators, the usual solutions of these tasks
involve computing Gr\"obner bases and factoring multivariate polynomials over extension fields of~$K$
(see for instance~\cite{DGP}, \cite{GWW}, \cite{vzGG}, \cite{KR1}, \cite{KR2}, \cite{KR3}).

The difficulty of the problem increases further when we consider algebras over the integers, i.e., algebras
of the form $R=\Z[x_1,\dots,x_n]/I$ with an ideal~$I$ given by a system of generators.
In this case we will also have to factor (potentially large) integers, as already the example
$R=\Z/ n\Z$ shows. Since the 1970s, various approaches have been taken to tackle these tasks, starting
with the case of an algebra~$R$ which is a finitely generated $\Z$-module (see~\cite{AL}, \cite{Ayo}, \cite{Se},
\cite{PSS}). At the core of most of these algorithms lies the calculation of strong Gr\"obner bases
for ideals in $\Z[x_1,\dots,x_n]$ which tends to be quite demanding.
It is also possible to apply more general algorithms for associative, not necessarily
commutative algebras here (see~\cite{DK}, \cite{EG}, \cite{Ro}), but we can expect the efficiency
of such very general methods to be usually even lower than the ones for commutative algebras.

The situation changes substantially when a $\Z$-algebra~$R$ is {\it explicitly given},
i.e.\ a $\Z$-algebra for which we know a system of generators $G=(g_0,\dots,g_n)$ of its additive group, 
a system of generators of the $\Z$-linear relation module $\Syz_\Z(G) \subseteq \Z^{n+1}$, 
and the {\it structure constants} $c_{ijk}\in\Z$ such that $g_i \cdot g_j = \sum_{k=0}^n c_{ijk} g_k$ 
for $i,j=0,\dots,n$. 

To define algebras by their module generators and relations, as well as their structure constants,
is a very classical approach, followed for instance by Bourbaki in~\cite{Bou}, Ch.~III, \S 1, Sect.~7.
Knowing the structure constants is equivalent to knowing the multiplication matrices of an algebra.
In particular, when the algebra is defined over a field, this point of view is one of the key methods to
study 0-dimensional affine algebras and to solve 0-dimensional polynomial systems 
(see for instance~\cite{Cox} and~\cite{KR3}).

In a recent paper~\cite{KMW}, we encountered explicitly given $\mathbb{Z}$-algebras in a different way:
when a non-commutative algebra is given by representing its left and right multiplications via 
endomorphisms, one can efficiently compute a ring of scalars. As a result, the ring of scalars 
is an explicitly given, finite commutative $\Z$-algebra.
For further examination of the original rings, it is necessary to find algorithms that perform
some of the operations mentioned above on the ring of scalars. In~\cite{KMW} we formulated a 
few of those algorithms using the calculation of strong Gr\"obner bases.
Here we avoid Gr\"obner bases and provide precise worst-case complexity bounds which are {\it almost}
polynomial time.

Thus the main task tackled in this paper can be described as follows: assume that a finite $\Z$-algebra~$R$
is given explicitly, i.e., by generators, relations, and structure constants. Develop algorithms
for computing their nilradical, associated primes and maximal ideals, primitive idempotents and connected components
of $\Spec(R)$ with the lowest possible worst-case complexity. More precisely, we shall show that
all these tasks can be solved in ZPPIF, i.e., in zero-error probabilistic polynomial time plus possibly 
one integer factorization.

Let us discuss the contents of the paper in more detail.
Throughout we work with an explicitly given finite $\Z$-algebra~$R$ defined as above.
In Section~2 we start by using the well-known facts that the Smith and Hermite normal forms of an 
integer matrix can be calculated in polynomial time (see~\cite{KaBa}, \cite{Laz}, \cite{Sj1}, \cite{Sj2}), 
in order to construct a polynomial time algorithm for solving linear systems of equations over~$R$
(see Proposition~\ref{prop:linear_sys}).
Based on this algorithm we perform various operations with ideals in~$R$ efficiently, for instance 
ideal intersections (see Proposition~\ref{prop:elementary_computations}). Moreover, we are able
to compute preimages under the isomorphism given by the Chinese Remainder Theorem 
(see Proposition~\ref{prop:preimage}).

In Section~3 we prepare later applications to $\Z$-algebras by reconsidering and reanalyzing some
algorithms for 0-dimensional algebras over a field~$K$. In order to compute the nilradical
of an explicitly given $K$-algebra~$R$, we need to calculate the factorization of univariate polynomials 
over~$K$. The currently best algorithms have polynomial time complexity (P) in the case $K=\Q$
and zero-error probabilistic polynomial time complexity (ZPP) in the case of a prime field~$\F_p$.
Using such polynomial factorization algorithms, Algorithm~\ref{alg:radical} then determines the nilradical 
of~$R$ with these time complexities (P resp.\ ZPP).

Next we want to compute the primary decomposition of the zero ideal of~$R$. In the case $K=\F_p$, we apply
the method of Frobenius spaces (see~\cite{KR3}, Alg.\ 5.2.7) and get an algorithm 
in ZPP (see Algorithm~\ref{alg:primary_decomp}). In the case $K=\Q$, we can use the method of~\cite{KR3}, Alg.~5.4.2
and get an algorithm in~P.
Altogether, by applying the primary decomposition algorithms to $R/\Rad(0)$, 
we are able to find the maximal ideals of~$R$ 
in~P and ZPP for $K=\Q$ and $K=\F_p$, respectively (see Corollary~\ref{cor:complexityofmax}).

In Section~4 we then use the ideas of~\cite{PSS} to compute the associated primes of an explicitly given
finite $\Z$-algebra~$R$ (see Algorithm~\ref{alg:ass_primes}). The method is to distinguish between the prime ideals
which contain an integer prime number and those which don't. In both cases the computation is reduced to
the setting of the preceding section, i.e., to computations of 0-dimensional algebras over fields.
Since the determination of associated primes may involve the factorization of
an integer, the best time complexity we can achieve here is ZPPIF, i.e., ZPP plus one integer factorization.
More precisely, the integer which has to be factored is the torsion exponent of the additive group of~$R$.

In Section~5 we treat the next topics, namely the computation of the primitive idempotents of an explicitly given
finite $\Z$-algebra~$R$ and the connected components of its prime spectrum.
It turns out that it is advisable to solve the latter task first. 
More precisely, since~$R$ may have infinitely many prime ideals, we compute the
connected components of the set of minimal associated primes $\minPrimes(R)$ in Algorithm~\ref{alg:connected}.
This algorithm uses the results of the preceding section, whence its worst-case time complexity is in ZPPIF.
Finally, we calculate the primitive idempotents of~$R$ in Algorithm~\ref{alg:comp_prim_idemp}
by lifting them from the primitive idempontents of $R/\Rad(0)$. This lifting process is performed using 
Algorithm~\ref{alg:lifting}. Once again the worst-case time complexity is in ZPPIF.

In the last section we connect the method of using an explicit representation of~$R$
to the more traditional method of calculating a strong Gr\"obner basis, when~$R$ is given as
$R=P/I$ with $P=\Z[x_1,\dots,x_n]$ and an ideal~$I$ in~$P$ whose generators are known.
Starting with an explicitly given finite $\Z$-algebra~$R$,
we can compute in polynomial time a strong Gr\"obner basis of a defining ideal~$I$ of~$R$ 
(see Corollary~\ref{cor:CompStrongGB}). For this direction we use a generalization
of the Buchberger-M\"oller Algorithm (see Algorithm~\ref{alg:BMintersection}).
Conversely, starting with a strong Gr\"obner basis of~$I$, we can calculate
an explicit representation of~$R$ in polynomial time (see Algorithm~\ref{alg:module_presentation}).
For this direction, we use a generalization of Macaulay's Basis Theorem to finite $\Z$-algebras
(see Proposition~\ref{prop:macaulay}). 

For the notation and basic definitions we adhere to conventions in~\cite{KR1} and~\cite{KR2}.
All algorithms in this paper were implemented in the computer algebra system ApCoCoA (see~\cite{ApCoCoA})
and are available from the authors upon request. These implementations were used to calculate the examples
given in the various sections.

\bigskip\bigbreak
%
%

\section{Polynomial Time Computations in Finite $\Z$-Algebras}
\label{sec2}

Let $R$ be a finite $\Z$-algebra, i.e., a $\Z$-algebra which is a finitely generated $\Z$-module. 
We denote the additive group of $R$ by $R^+$.
In this section we collect operations in $R$ which can be computed in polynomial time if a presentation of $R$
is given as below.

\begin{rem}{\bf (Explicitly Given $\Z$-Algebras)}\label{remark:input}\\
Subsequently we assume that a $\Z$-algebra $R$ is given by the following information.
\begin{enumerate}
\item[(a)] A set of generators $\mathcal{G} = \{g_0, \dots, g_n\}$ of the $\Z$-module $R^+$, together with
a matrix $A = (a_{\ell k}) \in \Mat_{m,n+1}(\Z)$ whose rows generate the syzygy module
$\Syz_{\Z}(\mathcal{G})$ of $\mathcal{G}$.

\item[(b)] Structure constants $c_{ijk}$ such that $g_i g_j = \sum_{k=0}^n c_{ijk} g_k$ for $i,j = 0, \dots, n$.
\end{enumerate}
Notice that we can assume that $g_0 = 1$, and encode this information as an ideal
$$
I = \langle x_i x_j - \ssum_{k=0}^n c_{ijk} x_k, \; \ssum_{k=0}^n a_{\ell k} g_k
\mid i,j=1, \dots, n, \; \ell=1, \dots, m \rangle
$$
in $P = \Z[x_1, \dots, x_n]$ such that $R \cong P/I$.
If $R$ is given as above, we call it an \textbf{explicitly given $\Z$-algebra}.
\end{rem}

The bit complexity of the matrix $A$ in~(a) which defines the $\Z$-module structure of $R^+$ is given by
$$
\beta = (n+1)m \log_2(\norm{A}) \quad \text{with} \quad \norm{A} = \max\{\lvert a_{\ell k} \rvert \}.
$$
The bit complexity of the entire input defining the $\Z$-algebra $R$ is then given by
$$
\gamma = ((n+1)^3+(n+1)m) \log_2(M) \quad \text{with} \quad M = 
\max\{\lvert a_{\ell k}\rvert, \lvert c_{ijk} \rvert\}.
$$

In this section we collect computations in $R$ which can be performed in polynomial time in~$\beta$
or~$\gamma$, respectively. More precisely, we will use the following complexity classes.

\begin{defi}{\bf (Polynomial Time Complexity Classes)}\\
Consider an algorithm which takes a tuple of integers as input.
\begin{enumerate}
\item[(a)] The algorithm is in the complexity class \textbf{P (polynomial time)} if its running time is bounded by a
polynomial expression in the bit complexity of the input.

\item[(b)] The algorithm is in the complexity class \textbf{ZPP (zero-error probabilistic polynomial time)} if it is
a Las Vegas algorithm which has polynomial running time in the bit complexity of the input.

\item[(c)] The algorithm is in the complexity class \textbf{ZPPIF (zero-error probabilistic polynomial time plus
integer factorization)} if, except for the factorization of one integer, the algorithm is in ZPP and the
bit size of the integer to be factored is bounded by a polynomial expression in the bit complexity of the input.

\end{enumerate}
\end{defi}

It is useful to bring the $\Z$-module presentation of $R^+$ into a normal form.

\begin{rem}\label{remark:smith}
Let $A = (a_{ij}) \in \Mat_{m,n+1}(\Z)$ be the matrix whose rows are given by the generators
of~$\Syz(\mathcal{G})$. Then there exist unimodular transformation matrices $S \in \Mat_{m,m}(\Z)$ and
$T \in \Mat_{n+1,n+1}(\Z)$ such that
$$ 
S \cdot A \cdot T =
        \begin{pmatrix}
            k_1 & 0 & \cdots & \cdots & \cdots & 0 \\
            0 & \ddots & \ddots & & & \vdots \\
            \vdots & \ddots & k_u & \ddots & & \vdots \\
            \vdots & & \ddots & 0 & \ddots & \vdots \\
            \vdots & & & \ddots & \ddots & 0 \\
            0 & \cdots & \cdots & \cdots & 0 & 0
        \end{pmatrix}
$$
and $k_i$ divides $k_j$ for $i<j$. This matrix is called the \textbf{Smith normal form} of $A$.
It yields the following isomorphism:
$$
R^+ \cong \Z^r \oplus \Z/k_1 \Z \oplus \cdots \oplus \Z/k_u \Z.
$$
The numbers $r$ and $k_1, \dots, k_u$ are uniquely determined by $R^+$. We call $r$ the \textbf{rank} and
$k_1, \dots, k_u$ the \textbf{invariant factors} of~$R^+$. The largest invariant factor $k_u$ is the exponent
of the torsion subgroup of $R^+$. We call it the \textbf{torsion exponent}~$\tau$ of~$R^+$.
\end{rem}

In the following we shall show that certain algorithms run in polynomial time 
by reducing them to the following computations.

\begin{rem}{\bf (Complexity of Integer Linear Algebra Operations)}\label{remark:computs_over_Z}
\begin{enumerate}
\item[(a)] The Smith and the Hermite normal form of a matrix $A \in \Mat_{m,n}(\Z)$ can be computed in
polynomial time, as first shown by Kannan and Bachem \cite{KaBa} in 1979. Currently, the fastest
deterministic algorithm for computing the Smith normal form is the one developed by Storjohann~\cite{Sj1}.
Note that in contrast to~\cite{KaBa} this algorithm does not produce the unimodular transformation matrices.

\item[(b)] Solving linear systems of equations over the integers can be reduced to computing a Smith normal form
together with the unimodular transformation matrices (see for instance~\cite{Laz}). If the linear system
is of the form $Ax=b$ where $A \in \Mat_{m,n}(\Z)$ and $b \in \Mat_{n,1}(\Z)$,
then generators of the solution space can be computed in polynomial time in $n$, $m$, $\norm{A}$, $\norm{b}$
and the rank of $A$. Here $\norm{A}$ denotes the maximal absolute value of the entries of $A$.
A concrete complexity bound is given in~\cite{Sj2}, Theorem~19.

\item[(c)] Computing the intersection of free submodules of $\Z^n$ can be achieved by computing a basis of
the solution space of an appropriate linear system of equations. The problem therefore reduces to~(b).
            
\end{enumerate}
\end{rem}

Since the Smith normal form can be computed in polynomial time, it follows that the bit complexity of the torsion
exponent of~$R$ is bounded by a polynomial in~$\beta$. Below we give a concrete complexity bound.

\begin{lem}\label{lemma:exponent_bound}
Let $R$ be an explicitly given finite $\Z$-algebra. Then the bit complexity of the torsion exponent $\tau$
is bounded by $n \log_2(n\norm{A})$.
\end{lem}

\begin{proof}
The product of the invariant factors of $R^+$ is given by the gcd of all maximal rank minors of $A$. The torsion
exponent is therefore bounded by the absolute value of a non-zero maximal rank minor of $A$. Hadamard's inequality
then yields $\tau \leq n^{n/2} \norm{A}$, which means the bit complexity of $\tau$ is bounded by $n \log_2(n\norm{A})$.
\end{proof}

Solving a linear system of equations over $R$ can be reduced to solving a linear system over the integers.

\begin{prop}{\bf (Solving Systems of Linear Equations over~$R$)}\label{prop:linear_sys}\\
Let $R$ be an explicitly given finite $\Z$-algebra and $f_1,\dots,f_p\in R$. For $k=1,\dots,p$, we write
$f_k = b_{k0}g_0 + \cdots + b_{kn} g_n$ with $b_{kj}\in\Z$. Let $y_1,\dots,y_p$ be further indeterminates.
Consider the following homogeneous linear equation over~$R$.
\begin{equation*}
\tag{i}\label{eqn_over_R}         f_1 y_1 + \cdots + f_p y_p \;=\; 0
\end{equation*}
Let $e_0, \dots, e_n \in \Z^{n+1}$ be the standard basis vectors.
For the following system of homogeneous linear equations in the indeterminates $z_{ki}$ and $w_j$ over~$\Z$,
let $\mathcal{L}$ be the projection of the solution space onto the $z$-coordinates.
\begin{equation*}
\tag{ii}       \ssum_{k=1}^p \ssum_{i,j,\ell=0}^n z_{ki}b_{kj}c_{ij\ell} e_\ell
               - \ssum_{j=1}^m \ssum_{i=0}^n w_j a_{ij}e_i \;=\; 0
\end{equation*}
Then the following conditions are equivalent.
\begin{enumerate}
\item[(a)] A tuple $(h_1,\dots,h_p) \in R^p$ with $h_k = d_{k0}g_0 + \cdots + d_{kn}g_n \in R$ and $d_{ki} \in \Z$ 
is a solution of~(\ref{eqn_over_R}). 

\item[(b)] The tuple $(d_{ki})$ is an element of~$\mathcal{L}$.
\end{enumerate}
\end{prop}

\begin{proof}
The tuple $(h_1,\dots,h_p)$ is a solution of~(i) if and only if
$$
\ssum_{k=1}^p \ssum_{i,j=0}^n b_{ki}d_{kj}g_ig_j=0.
$$
This is the case if and only if there exist $\alpha_{1},\dots,\alpha_{m}\in\Z$ such
that the left hand side is equal to $\sum_{j=1}^m \sum_{i=1}^n \alpha_j a_{ij}g_i$. 
The claim then follows by rewriting the products~$g_ig_j$ using the structure constants of~$R$ 
and applying the canonical isomorphism $R \cong \Z^{n+1}/\Syz(G)$.
\end{proof}

Let us illustrate this proposition with an example.

\begin{exa}
Let $R$ be the finite $\mathbb{Z}$-algebra generated by $\mathcal{G} = \{g_1, g_2, g_3\}$, where
$\Syz(\mathcal{G}) = \langle (3,0,0), (-1,0,4) \rangle$, and where
the multiplication in~$R$ is commutative and given by
$g_1^2 = 3g_1$, $g_1g_3 = 2g_2$, $g_2^2 = g_1+g_2$, and $g_ig_j = 0$ for all other combinations.
Consider the homogeneous linear equation
$$
  (2g_3)\, x_1 + (g_1+g_3)\, x_2 + (2g_1)\,x_3 \;=\; 0
$$
over~$R$.
Every solution is of the form $(h_1, h_2, h_3) \in R^3$ with $h_i = d_{i1} g_1 + d_{i2} g_2 + d_{i3} g_3$,
where $d_{ij}\in \mathbb{Z}$.
To compute generators of the solution space $\mathcal{L} \subseteq R^3$, we follow Proposition~\ref{prop:linear_sys}
and substitute~$x_i$ by~$h_i$. Using the structure constants, we then replace products $g_i g_j$ by
$\mathbb{Z}$-linear combinations of the generators and obtain the system of equations
\begin{align*}
    4 d_{11} g_2 = 0 \\
    3 d_{21} g_1 + 2 d_{23} g_2 + 2 d_{21} g_2 = 0 \\
    6 d_{31} g_1 + 4 d_{33} g_2 = 0.
\end{align*}
By substituting $g_i$ with $e_i$ and taking into account the generators of $\Syz(\mathcal{G})$, we obtain
a system of linear equations over $\mathbb{Z}$ which is given by the following matrix.
\setcounter{MaxMatrixCols}{15}
$$
\begin{pmatrix}
     0 & 0 & 0 & 3 & 0 & 0 & 6 & 0 & 0 & 3 & -1 \\
     4 & 0 & 0 & 2 & 0 & 2 & 0 & 0 & 4 & 0 & 0 \\
     0 & 0 & 0 & 0 & 0 & 0 & 0 & 0 & 0 & 0 & 4
\end{pmatrix}
$$
After projecting onto the first nine components of its solution space we obtain for example the
tuple $(0,0,0,1,0,1,0,0,-1) \in \mathbb{Z}^9$ which corresponds to the tuple $(0,g_1+g_2,-g_3) \in R^3$. 
The whole solution space in~$R^3$ is generated by this tuple together with seven further tuples.
\end{exa}

Proposition~\ref{prop:linear_sys} yields the following complexity bound for solving linear equations.

\begin{cor}
Let $R$ be an explicitly given finite $\Z$-algebra. Generators of the solution space of a linear
equation over $R$ as in Proposition~\ref{prop:linear_sys} can be computed in polynomial time in the bit complexity
of the input which is given by $\gamma$ (for~$R$) and by $p(n+1) \log_2(M)$ where $M = \max\{b_{kj}\}$ 
(for the elements $f_1,\dots,f_p)$.
\end{cor}

\begin{proof}
The coefficients in the system of equations (ii) in Proposition~\ref{prop:linear_sys} are $b_{kj}$, $c_{ij\ell}$
and $a_{ij}$. The claim therefore follows immediately from Remark~\ref{remark:computs_over_Z}.b.
\end{proof}

The next proposition collects elementary operations in an explicitly given finite $\Z$-algebra 
which can be performed in polynomial time.

\begin{prop}{\bf (Elementary Ideal-Theoretic Operations)}\label{prop:elementary_computations}\\
Let $R$ be an explicitly given finite $\Z$-algebra, and let
$J = \langle f_1, \dots, f_k \rangle$ as well as $J' = \langle h_1, \dots, h_\ell \rangle$
be ideals in~$R$. We assume that the elements $f_i, h_j\in R$ are given as elements 
in $\Z[g_0,\dots,g_n]$ and that the bit complexity of these sets of polynomials is given 
by~$\delta_J$ and~$\delta_{J'}$, respectively.
\begin{enumerate}
\item[(a)] The rank and the invariant factors of $R^+$ can be computed in polynomial time in~$\beta$.

\item[(b)] Let $\overline{\mathcal{G}} = \{ \overline{g}_0, \dots, \overline{g}_n\}$ be the set of 
residue classes in $R/J$ of the elements of~$\mathcal{G}$. Then generators of $\Syz_{\Z}(\overline{\mathcal{G}})$ 
can be computed in polynomial time in~$\gamma+\delta_J$.

\item[(c)] We can decide whether $J \subseteq J'$ in polynomial time in $\gamma+\delta_J+\delta_{J'}$.

\item[(d)] We can decide whether $J = \langle 1 \rangle$ in polynomial time in $\gamma+\delta_J$.

\item[(e)] Generators of the intersection $J \cap J'$ can be computed in polynomial time
in $\gamma+\delta_J+\delta_{J'}$.
\end{enumerate}
\end{prop}

\begin{proof}
To prove (a), let $A = (a_{ij}) \in \Mat_{n+1, m}(\Z)$ be the matrix whose rows are given by the
generators of~$\Syz_{\Z}(\mathcal{G})$. The rank and the invariant factors of~$R^+$ can be determined from the Smith
normal form of~$A$, which can be computed in polynomial time by Remark~\ref{remark:computs_over_Z}.

Next we show (b). Using the structure constants, we can rewrite the $f_i$ as linear combinations
$b_{i0}g_0+ \cdots + b_{in} g_n$ of the generators of~$R^+$. This means that we obtain integer tuples
$v_1, \dots, v_r \in \Z^{n+1}$ with $v_i = (b_{i0}, \dots, b_{in})$ such that
$v_1, \dots, v_r$ together with the generators of $\Syz_{\Z}(\mathcal{G})$ generate
$\Syz_{\Z}(\overline{\mathcal{G}})$.

Using~(b), we can compute presentations $R/J \cong \Z^{n+1}/V_1$ and
$R/J' \cong \Z^{n+1}/V_2$, where~$V_1$ and~$V_2$ are submodules of $\Z^{n+1}$, in polynomial time.
The ideal~$J$ is then contained in~$J'$ if and only if $V_1 \subseteq V_2$. This proves~(c).

To show~(d), we use part~(b) to compute a presentation
$R/J \cong \Z^{n+1}/\Syz_{\Z}(\overline{\mathcal{G}})$. We can then apply~(a) and compute the rank
and the invariant factors of $R/J$ in polynomial time. Notice that we have $J = \langle 1 \rangle$ 
if and only if the rank of $R/J$ is zero and all invariant factors are equal to one.

For the proof of~(e), we let 
$$
\mathcal{M} =  \begin{pmatrix}
        1 & f_1 & \cdots & f_k & 0   & \cdots & 0\\
        1 & 0   & \cdots & 0      & h_1 & \cdots & h_\ell
\end{pmatrix}.
$$
Generators of $\Syz_R(M)$ can be computed in polynomial time by solving an appropriate linear system of equations
over~$R$ using Proposition~\ref{prop:linear_sys}. The first non-zero coordinates of the generators then generate
$J \cap J'$ by \cite{KR1}, Proposition~3.2.3.
\end{proof}

The following algorithm will come in handy when we compute the primitive idempotents of a finite $\Z$-algebra.

\begin{algorithm}{\bf (The Chinese Remainder Preimage Algorithm)}\label{prop:preimage}\\
Let $R$ be an explicitly given finite $\Z$-algebra. In particular, we assume that $R^+$ is generated by
$\mathcal{G} = \{g_0, \dots, g_n\}$ with $g_0 =1$. Let $J_1, \dots, J_s$ be pairwise comaximal ideals
in~$R$, and assume that $J_1 \cap \cdots \cap J_s = \langle 0 \rangle$. 
Given $i\in \{1,\dots,s\}$, consider the following sequence of instructions.
\begin{enumerate}
\item[(1)] Using Proposition~\ref{prop:elementary_computations}.b, compute $\Z$-submodules
$V_j \subseteq \Z^{n+1}$ such that we have $R/J_j \;\cong\; \Z^{n+1}/V_j$ for $j=1,\dots,s$.

\item[(2)] Compute a $\Z$-module basis $\{v_1, \dots, v_k \} \subseteq \Z^{n+1}$
of~$\bigcap_{j\neq i} V_j$.

\item[(3)] Let $\{w_1, \dots, w_\ell\}\subseteq \Z^{n+1}$ be a $\Z$-basis of~$V_i$. Compute a solution
$(c_i) \in \Z^{k+\ell}$ of the linear system of equations in the indeterminates
$y_1, \dots, y_{k+\ell}$ given by
$$
v_1 y_1 + \cdots + v_k y_k = w_1 y_{k+1}+ \cdots + w_\ell y_{k+\ell} + (1,0, \dots, 0).
$$

\item[(4)] Let $h = (h_0,\dots,h_n) = c_1 v_1 + \cdots + c_kv_k \in \Z^{n+1}$. Return the
element $f = h_0 g_0+\cdots + h_n g_n$ of~$R$ and stop.
\end{enumerate}
This is a polynomial time algorithm which computes an element $f \in R$ such that~$f$ is mapped to
the $i$-th canonical basis vector~$e_i$ under the canonical $\Z$-linear map
$$
\varphi:\; R \;\longrightarrow\;   R/J_1 \times \cdots \times R/J_s.
$$
\end{algorithm}

\begin{proof}
The tuple $h$ satisfies $h \in \bigcap_{j\neq i} V_j$ and $h -(1,0, \dots, 0) \in V_i$. This shows
that the residue class of $h$ in $\Z^{n+1}/\Syz_{\Z}(\mathcal{G})$ is mapped to $e_i$ under the
canonical map
$\psi:\; \Z^{n+1}/\Syz_{\Z}(\mathcal{G}) \longrightarrow \Z^{n+1}/V_1 \times \cdots \times
\Z^{n+1}/V_s$.
Hence~$f$ is mapped to~$e_i$ under the map~$\varphi$. Steps~(1) and~(2) of the algorithm can be performed in
polynomial time by Proposition~\ref{prop:elementary_computations}. The linear system in Step~(3) can also be solved
in polynomial time by Remark~\ref{remark:computs_over_Z}.
\end{proof}

Let us apply this algorithm to a concrete case.

\begin{exa}
Consider the finite $\mathbb{Z}$-algebra $R = \mathbb{Z}[x,y]/\langle x^3 + x^2, 3x^2 + 3x, xy + y, y^2, 2y \rangle$.
It is generated as a $\mathbb{Z}$-module by the elements of $\mathcal{G} = (1, \bar{x}^2, \bar{x}, \bar{y})$, and
$\Syz_{\mathbb{Z}}(\mathcal{G})$ is generated by $(0,0,0,2), (0,3,3,0)$. Consider the ideals
$J_1 = \langle \bar{y}^2, \bar{x}+1, 2\bar{y} \rangle$ and $J_2 = \langle \bar{x}^2, 3\bar{x}, \bar{y} \rangle$
in~$R$. Our goal is to compute $f \in R$ such that $f$ is mapped to~$e_2$ under the canonical $\mathbb{Z}$-linear
map $R \rightarrow R/J_1 \times R/J_2$.

\begin{enumerate}
\item[(1)] Using Proposition~\ref{prop:elementary_computations}.b, we find
$V_1 = \langle (1,0,1,0), (0,0,0,2), (-1,1,0,0) \rangle$ and $V_2 = \langle (0,0,3,0), (0,1,0,0), (0,0,0,1) \rangle$
such that $R/J_1 \cong \mathbb{Z}^4/V_1$ and $R/J_2 \cong \mathbb{Z}^4/V_2$.

\item[(2)] We have $\bigcap_{j\neq 2} V_j = V_1$.

\item[(3)] A solution of the linear system
$$
\begin{pmatrix}
   1 & 0 & -1 & 0 & 0 & 0 & -1 \\
   0 & 0 & 1 & 0 & -1 & 0 & 0 \\
   1 & 0 & 0 & -3 & 0 & 0 & 0 \\
   0 & 2 & 0 & 0 & 0 & -1 & 0
\end{pmatrix}
\cdot
\begin{pmatrix} y_1 \\ \vdots \\ y_7 \end{pmatrix}
= \begin{pmatrix} 1  \\ 0  \\ 0  \\ 0 \end{pmatrix}
$$
is given by $(3,0,2,1,2,0) \in \mathbb{Z}^6$.

\item[(4)] This solution yields the tuple $h = (1,2,3,0)$ which corresponds to the element
  $f = 1+3\bar{x}+2\bar{x}^2$ in $R$. It is the preimage of $e_2$ under the canonical
  map $R \rightarrow R/J_1 \times R/J_2$.

\end{enumerate}
\end{exa}

\bigskip\bigbreak
%
%

\section{Computing the Maximal Ideals of a 0-Dimensional $K$-Algebra}
\label{sec3}

In this section we assume that $K$ is either the field of rational numbers $\Q$ or a finite field~$\F_p$. 
Our goal is to study the complexity of computing the maximal components of a 0-dimensional 
$K$-algebra~$R$ which is explicitly given in the following sense.

\begin{defi}
A 0-dimensional $K$-algebra $R$ is \textbf{explicitly given} if it is given by a $K$-vector space
basis $\mathcal{B} = \{b_1, \dots, b_n\}$ and structure constants $c_{ijk}$ such that
$b_ib_j = \sum_{k=1}^n c_{ijk} b_k$ for all $i,j=1, \dots, n$.
\end{defi}

Note that a 0-dimensional $K$-algebra as in this definition can equivalently be given by a basis together with
multiplication matrices. The crucial step in computing the maximal ideals of~$R$ is the factorization of
univariate polynomials over~$K$.

\begin{rem}\label{remark:factoring}
In 1982 Lenstra et al.~\cite{LLL} published a deterministic algorithm for factoring univariate polynomials in
$\Q[x]$. The running time of their algorithm is polynomial in $\deg(f)$ and $\log(\lvert f\rvert )$, where for a
polynomial $f = \sum_i a_i x^i \in \Q[x]$ we define $\lvert f\rvert = \sqrt{\sum_i a_i^2}$. This means it requires
only a polynomial number of bit operations measured in the input size.

For univariate polynomials over finite fields, the situation is slightly more complicated. A deterministic
algorithm for factoring polynomials over finite fields was presented by Berlekamp in~\cite{Be1}. Its running time
for factoring a polynomial $f \in \F_p$ is polynomial in $p$ and $\deg(f)$. But this is not polynomial
in the bit complexity of the input which is given by $(1+\deg(f)) \log_2(p)$. In 1970 Berlekamp published
a Las Vegas algorithm~\cite{Be2} for the problem which has polynomial running time in the input. Since then many
new and faster algorithms were developed, see e.g.~\cite{vzG}.
But it is still unknown whether the factorization can be performed in deterministic polynomial time. It was shown by
Evdokimov \cite{Ev} that under the generalized Riemann hypothesis (GRH) the problem can be solved in subexponential time. 
Furthermore, there have been efforts to drop the GHR assumption (see~\cite{IKRS}).
In addition, there exist deterministic polynomial time algorithms~\cite{Ga, Sa} for many special classes of
polynomials over finite fields. Indeed, it is conjectured that the set of polynomials which do not satisfy any of
the conditions in~\cite{Sa} is empty.
\end{rem}

The first step in computing the maximal ideals of~$R$ is to compute its nilradical.

\begin{algorithm}{\bf (Computing the Nilradical of a 0-Dimensional Algebra)}\label{alg:radical}\\
Let $R$ be an explicitly given 0-dimensional $K$-algebra. Consider the following sequence of instructions.
\begin{enumerate}
\item[(1)] Let $J=\langle 0 \rangle$ and $\mathcal{B} = \{b_1, \dots, b_n\}$.

\item[(2)] For $i=1,\dots,n$, perform the following steps (3)-(7).

\item[(3)] Compute the minimal polynomial $\mu_{b_i+J}(z)$ of $b_i +J$ in $R/J$.

\item[(4)] Calculate $g_i(z) = \sqfree(\mu_{b_i+J}(z))$.

\item[(5)] Replace $J$ with $J + \langle g_i(b_i) \rangle$.

\item[(6)] Using the structure constants, rewrite $g_i(b_i)$ and try to obtain a representation of some
    $b_j \in \mathcal{B}$ as a linear combination of the remaining elements. If such linear combinations exist remove
    those elements $b_j$ from $\mathcal{B}$ and update the structure constants to obtain
    an explicit presentation of~$R/J$.

\item[(7)] If $\deg(g_i(z)) = \dim_K(R/J)$, return the ideal~$J$ together with the explicit presentation of~$R/J$
and stop.

\item[(8)] Return the ideal~$J$ together with the explicit presentation of~$R/J$ and stop.
\end{enumerate}
This is an algorithm which computes the nilradical $\Rad(0)$ of~$R$ together with an explicit presentation of $R/\Rad(0)$.
If $K = \Q$, or if $K$ is a finite prime field, then it has polynomial running time. In particular, the bit
complexity of the explicit representation of $R/\Rad(0)$ is polynomially bounded by the bit complexity of the input.
\end{algorithm}

\begin{proof}
The correctness of this algorithm is shown in \cite{KR3}~Algorithm~5.4.2. It remains to prove that it runs in
polynomial time. The minimal polynomial in step~(3) can be computed using \cite{KR3}, Algorithm~1.1.8. It involves
finding linear dependencies among the elements $1+J$, $b_i+J$, $\dots$, $b_i^d+J$ where $d = \dim_K(R/J)$. 
Using the structure constants, we rewrite $b_i^j$ for $j=2,\dots,d$ as linear combinations of the elements of~$\mathcal{B}$.
The linear dependencies can then clearly be found in polynomial time. The squarefree part of $g_i(z)$ in step~(4)
can also be computed in polynomial time (see~\cite{vzGG} Section~14.6).

The bit complexity of the presentation of $R/\Rad(0)$ is polynomially bounded by the bit complexity of the input, 
since during each iteration the bit complexity of the structure constants obtained in step~(6) is polynomially bounded 
by the bit complexity of the structure constants of the previous iteration.
\end{proof}

The following example illustrates how this algorithm can be applied.

\begin{exa}
Consider the zero-dimensional $\mathbb{Q}$-algebra $R$ with basis $\{1,b_1,b_2,b_3\}$ and multiplication given by
$b_1^2 = 2b_1-1$, $b_1b_2 = b_3$, $b_1b_3 = 2b_3-b_1$, $b_2^2 = -b_2-1$, $b_2b_3= -b_3-b_1$ and
$b_3^2 = -2b_3+b_2-2b_1+1$.
\begin{enumerate}
\item[(1)] We let $J = \langle 0 \rangle$ and $\mathcal{B} = \{b_1, b_2, b_3,1\}$.

\item[(3)] The minimal polynomial $\mu_{b_1}(z)$ of $b_1$ is $z^2-2z+1$.

\item[(4)] We calculate $g_1(z) = \sqfree(\mu_{b_1}(z)) = z-1$.

\item[(5)] We set $J = \langle b_1-1 \rangle$.

\item[(6)] Substituting $b_1 = 1$ into $b_1b_2 = b_3$ yields $b_2 = b_3$. Therefore we set
$\mathcal{B} = \{b_2,1\}$ and obtain $\dim_{\mathbb{Q}}(R/J) = 2$.

\item[(3)] The minimal polynomial $\mu_{b_2+J}(z)$ of $b_2+J$ is $z^2+z+1$.

\item[(4)] We calculate $g_2(z) = \sqfree(\mu_{b_2}(z)) = \mu_{b_2}(z)$.

\item[(5)] We set $J = \langle b_1-1, b_2^2+b_2+1\rangle$.

\item[(6)] Rewriting $b_2^2+b_2+1$ using $b_2^2 = -b_2-1$, we only obtain the trivial relation given
by $0=0$. Thus $\mathcal{B}$ is not updated and $\dim_{\mathbb{Q}}(R/J) = 2$.

\item[(7)] Since $\dim_{\mathbb{Q}}(R/J) = 2 = \deg(g_2)$ we return the ideal
$J = \langle b_1-1, b_2^2+b_2+1\rangle$ together with the $\mathbb{Q}$-basis $\{1, \bar{b}_2\}$ of $R/J$
and the structure constant $b_2^2 = -b_2-1$.

\end{enumerate}
\end{exa}

Having computed the nilradical of~$R$, we can then obtain its maximal ideals as follows.
In the case $K = \Q$, we can use Algorithm 7.2 in~\cite{LS}. It has polynomial running time in the bit
complexity of the input. For $K = \F_p$, we can only hope for an algorithm in ZPP since we need to factor
univariate polynomials over~$\F_p$. In the more general case of associative algebras over finite fields, the
complexity of computing their structure, i.e., their simple components was studied in~\cite{FR}, \cite{Ro},
and~\cite{EG}. But let us take advantage of the fact that we are in the commutative case and analyze
the complexity of the algorithm presented in~\cite{KR3}, which was inspired by~\cite{GWW}. In contrast to the methods
cited above it has the advantage of being well-suited for an actual implementation.

\begin{defi}
Let $R$ be a 0-dimensional $\F_p$-algebra.
\begin{enumerate}
\item[(a)] The map $\phi_p :\; R \longrightarrow R$ defined by $a \mapsto a^p$ is an $\F_p$-linear ring
endomorphism of~$R$. It is called the \textbf{Frobenius endomorphism} of~$R$.

\item[(b)] The $\F_p$-vector subspace
$$
\Frob_p(R) = \{ f \in R \mid f^p - f = 0 \}
$$
of~$R$, i.e., the fixed-point space of~$R$ with respect to~$\phi_p$, is called the \textbf{Frobenius space} of~$R$.

\end{enumerate}
\end{defi}

In \cite{KR3}, Algorithm~5.2.7, it is explained how one can calculate the Frobenius space
of a 0-dimensional $\F_p$-algebra. Based on this result, we obtain the following algorithm.

\begin{algorithm}{\bf (Primary Decomposition in Characteristic $p$)}\label{alg:primary_decomp}\\
Let $R$ be an explicitly given 0-dimensional $\F_p$-algebra. In particular, we assume
that $\mathcal{B} = \{b_1, \dots, b_n\}$ is a $K$-vector space basis of~$R$. Consider the following sequence of
instructions.
\begin{enumerate}
\item[(1)] Form the multiplication matrix $M_{\mathcal{B}}(\phi_p)$ of the Frobenius endomorphism of $R$, and compute
the number $s = n -\rank(M_{\mathcal{B}}(\phi_p)-I_n)$ of primary components of the zero ideal of~$R$.
If $s=1$ then return $\langle 0 \rangle$ and stop.

\item[(2)] Let $L$ be the list consisting of the pair $(\langle 0 \rangle,s)$. Repeat steps~(3)--(6) until the
second component of all pairs in~$L$ is~1. Then return the tuple consisting of all first components of the
pairs in~$L$ and stop.

\item[(3)] Choose the first pair $(J,t)$ in $L$ for which $t>1$ and remove it from $L$.

\item[(4)] Using Algorithm~5.2.7 in \cite{KR3}, compute the Frobenius space of $R/J$. Choose a non-constant
element~$f$ in it.

\item[(5)] Calculate the minimal polynomial of the element $f$ and factor it in the form
$\mu_{f}(z) = (z-a_1) \cdots (z-a_u)$ with $a_1, \dots, a_u \in \F_p$.

\item[(6)] For $i=1, \dots, u$, let $J_i = J + \langle f-a_j \rangle$. Compute the dimension $d_i$ of
$\Frob_p(R/J_i)$ and append the pair $(J_i, d_i)$ to $L$.

\end{enumerate}
This is an algorithm which calculates the list of primary components of the zero ideal of~$R$. It is in ZPP.
\end{algorithm}

\begin{proof}
The correctness of this algorithm is shown in \cite{KR3}, Algorithm~5.2.11. In particular, it is proved there that
$t = d_1 + \cdots + d_u$ throughout the course of this algorithm. Therefore the number of iterations of steps (3)--(6) is
bounded by~$s$ which in turn is bounded by the vector space dimension~$n$ of~$R$. Algorithm~5.2.7 in step~(4) 
involves computing a basis for the kernel of a matrix over~$K$ and can therefore be done in polynomial time. As discussed in
the proof of Algorithm~\ref{alg:radical}, the minimal polynomial in step~(5) can also be computed in polynomial time. 
Computing its factorization is in ZPP by Remark~\ref{remark:factoring}.
\end{proof}

The following example shows this algorithm at work.

\begin{exa}
Consider the zero-dimensional $\mathbb{F}_2$-algebra~$R$ given by an $\mathbb{F}_2$-basis 
$B = \{1, b_1, b_2, b_3 \}$ and the multiplication $b_1^2 = b_1$, $b_1b_2 = b_3$, $b_1b_3 = b_3$, $b_2^2 = 1$, $b_2b_3 = b_1$, 
and $b_3^2 = b_1$.
\begin{enumerate}
\item[(1)] The structure constants provide for every $b \in B$ a representation of $b^2$ in terms of
the basis~$B$. This yields the matrix
$$
M_{\mathcal{B}}(\phi_2) =
\begin{pmatrix}
    1 & 0 & 1 & 0 \\
    0 & 1 & 0 & 1 \\
    0 & 0 & 0 & 0 \\
    0 & 0 & 0 & 0
\end{pmatrix},
$$
and we obtain $s = 4 -\rank(M_{\mathcal{B}}(\phi_2)-I_4) = 4-2 = 2$ for the number of primary components.

\item[(2)] Let $L = ((\langle 0 \rangle, 2))$.

\item[(3)] Choose the pair $(\langle 0 \rangle, 2)$, and let $L = ()$.

\item[(4)] The kernel of the matrix $M_{\mathcal{B}}(\phi_2)-I_4$ has the basis $\{(1,0,0,0),(0,1,0,0)\}$.
Therefore the Frobenius space of $R$ is given by $\langle 1, b_1 \rangle$. We choose $f = b_1$.

\item[(5)] The minimal polynomial of $f$ is given by $\mu_{f}(z) = z(z+1)$.

\item[(6)] Let $J_1 = \langle b_1 \rangle$ and $J_2 = \langle b_1+1 \rangle$. Using the structure constants
we see that the residue classes of $B' = \{1,b_2\}$ in $R/J_1$ and $R/J_2$ form a basis of the respective
algebras. In both cases we determine $s = 2 - \rank(M_{\bar{B}'}(\phi_2)-I_2) = 2-1 = 1$. Hence we
set $L = (J_1, 1), (J_2, 1))$.

\item[(2)] Since the second component of both pairs in $L$ is 1, we return the primary components
$\langle b_1 \rangle$ and $\langle b_1+1 \rangle$ of the zero ideal of $R$.

\end{enumerate}
\end{exa}

Using this algorithm, we can now calculate the maximal ideals of explicitly given 0-dimensional algebras.

\begin{cor}{\bf (Complexity of Computing the Maximal Ideals)}\label{cor:complexityofmax}\\
Let $K$ be the field of rational numbers or a finite prime field, and let~$R$ 
be an explicitly given 0-dimensional $K$-algebra.
\begin{enumerate}
\item[(a)] If $K = \Q$, then the maximal ideals of~$R$ can be computed in polynomial time.

\item[(b)] If $K = \mathbb{F}_p$, then the maximal ideals of~$R$ can be computed in ZPP.
\end{enumerate}
\end{cor}

\begin{proof}
Using Algorithm~\ref{alg:radical}, we compute the nilradical~$\Rad(0)$ of~$R$ in polynomial time. This
algorithm also yields an explicit presentation of $R/\Rad(0)$. If $K = \Q$, we then apply Algorithm~7.2
from~\cite{LS} to~$R/\Rad(0)$ and obtain the maximal ideals of~$R$ in polynomial time. Similarly, in the case 
$K = \mathbb{F}_p$, we apply Algorithm~\ref{alg:primary_decomp} to $R/\Rad(0)$.
\end{proof}

For further details and more examples which illustrate the algorithms presented in this section, we refer to Chapter~5
in~\cite{KR3}.

\bigskip\bigbreak
%
%

\section{Computing the Associated Primes of Finite $\Z$-Algebras}
\label{sec4}

In this section we let~$R$ be a finite $\Z$-algebra. We show that the associated primes of~$R$  can be computed in~ZPPIF,
if~$R$ is explicitly given. Note that the associated primes of~$R$ are given by the primary decomposition of its
nilradical~$\Rad(0)$. Algorithms for computing the primary decomposition of ideals in $\Z[x_1, \dots, x_n]$
date back to 1978~\cite{Ayo, Se}. More recently, Pfister et al.~\cite{PSS} presented a slightly different approach.
Inspired by this algorithm, we gave an efficient algorithm in~\cite{KMW} for computing the primary decomposition of
ideals $I \subseteq \Z[x_1, \dots, x_n]$ such that $P/I$ is a finite $\Z$-algebra. Let us now apply
this approach to explicitly given finite $\Z$-algebras.

The following lemma is used to split the computation into computing the associated primes
of 0-dimensional ideals in $\Q[x_1, \dots, x_n]$ and $\F_p[x_1, \dots, x_n]$.

\begin{lem}\label{lemma:splitting}
Let $R=P/I$ be an explicitly given finite $\Z$-algebra and let $\tau$ be its torsion exponent.
\begin{enumerate}
\item[(a)] The ideal $(I : \langle \tau \rangle)/I$ is the torsion subgroup of~$R^+$.

\item[(b)] We have $I = (I : \langle \tau \rangle) \cap (I+ \langle \tau \rangle)$.

\item[(c)] If $R$ is finite, then $I \cap \Z = \langle \tau \rangle$.
\end{enumerate}
\end{lem}

\begin{proof}
Part (a) follows immediately from the definition of the exponent of the torsion subgroup of~$R^+$. It then implies
$I:\langle \tau \rangle = I: \langle \tau \rangle^\infty$, which means that claim~(b) is a standard lemma in commutative
algebra. To prove (c), we note that the ring $P/I$ is finite if and only if there exists a positive integer $k \in \Z$
with $I \cap \Z = \langle k \rangle$. If such a number~$k$ exists, we have $k \cdot f = 0$ for all
$f \in R$, and therefore $\tau \mid k$. But we also have $n \cdot 1 = 0$ in $R = P/I$, and hence $\tau \in I$.
This implies $k \mid \tau$, and thus $k=\tau$.
\end{proof}

The associated primes of~$R$ can now be computed as described in the following algorithm.

\begin{algorithm}{\bf(Computing the Associated Primes)}\label{alg:ass_primes}\\
Let $R = P/I$ be an explicitly given finite $\Z$-algebra. Consider the following sequence of instructions.
\begin{algorithmic}[1]
   \STATE Set $L := [\;]$.
   \STATE Compute the torsion exponent $\tau$ of $R^+$.
   \IF {the rank of $R$ is not zero}
       \STATE Compute the prime components $\bar{\mathfrak{p}}_1 \cap \dots \cap \bar{\mathfrak{p}}_\ell$ of
       $I\,\Q[x_1, \dots, x_n]$.
       \STATE Compute $\bar{\mathfrak{p}}_j \cap P$ and append these ideals to~$L$.
       \STATE Recursively apply the algorithm to $I + \langle \tau \rangle$ and obtain the set~$M$.
       \STATE Compute $J := \bigcap_{{\scriptstyle\mathfrak{p}}\, \in L}\mathfrak{p}$.
       \STATE Remove all ideals in~$M$ that contain~$J$.
       \RETURN $L \cup M$
    \ELSE
       \STATE Compute all prime factors $p_1, \dots, p_r$ of $\tau$.
       \STATE Set $M := [\;]$.
       \FOR{$i=1,\dots,r$}
          \STATE Compute the prime components $\bar{\mathfrak{p}}_1 \cap \dots \cap \bar{\mathfrak{p}}_m$ of
          $I\,\F_{p_i}[x_1, \dots, x_n]$.
          \STATE Compute the preimages~$\mathfrak{\mathfrak{p}}_j$ of~$\bar{\mathfrak{p}}_j$ in~$P$ and append
          them to~$M$.
      \ENDFOR
      \RETURN $M$
   \ENDIF
\end{algorithmic}
This is an algorithm which computes the associated primes $\mathfrak{p}_1, \dots, \mathfrak{p}_k$ of~$R$. It is in~ZPPIF.
\end{algorithm}

\begin{proof}
The correctness of the algorithm follows from Lemma~\ref{lemma:splitting} and Proposition~4.7 in~\cite{KMW}.
Let us analyze the complexity of each of its steps. The torsion exponent and the rank of~$R$ can be computed in
polynomial time in~$\beta$ using Proposition~\ref{prop:elementary_computations}.a, and the bit complexity of
the torsion exponent is polynomially bounded by Lemma~\ref{lemma:exponent_bound}. Since $P/I$ is a finite
$\Z$-algebra, the ideals $I \Q[x_1, \dots, x_n]$ and $I \F_p[x_1, \dots, x_n]$ are
0-dimensional and therefore define 0-dimensional $\Q$- and $\F_p$-algebras, respectively. Their vector space
dimension is less than or equal to the number of generators of~$R$, and their structure constants are given by the
structure constants of~$R$. Thus we obtain the maximal components in lines~(4) and (14) in polynomial and probabilistic
polynomial time, respectively, by applying the algorithms in Section~\ref{sec3}. The intersection of the prime 
ideals in line~(7) can be computed in polynomial time by Proposition~\ref{prop:elementary_computations}.e. Finally,
Proposition~\ref{prop:elementary_computations}.c allows us to check the containment of ideals in line~(8) in
polynomial time.

In summary, all steps except for the integer prime factorization in line~(11) are in~ZPP.
\end{proof}

Note that, since the exponent~$\tau$ of~$R$ is the largest invariant factor of~$R$, all other invariant factors of~$R$
are divisors of~$\tau$. This means that we might already have a partial factorization of~$\tau$.
Let us compute the associated primes in a concrete case.

\begin{exa}
Consider the finite $\mathbb{Z}$-algebra $R$ given by the explicit presentation $R= \mathbb{Z}[x,y,z]/I$ 
with $I = \left\langle 6z, 6y, x^2+x-6, z^2, y^2, xy-y, xz-y, yz \right\rangle$. 
We follow the above algorithm and compute the associated primes of~$R$.
\begin{enumerate}
\item[(2)] Using Proposition~\ref{prop:elementary_computations}.a, we find that the torsion exponent 
of~$R$ is~6. 

\item[(4)]Since the rank of~$R$ is~$2$, we then compute the minimal associated prime ideals 
of $I \mathbb{Q}[x,y,z]$ using \cite{KR3}, Alg.~5.4.2 and obtain $\bar{\mathfrak{p}}_1 = \langle z, y, x + 3 \rangle$ 
as well as $\bar{\mathfrak{p}}_2 = \langle z, y, x -2 \rangle$.

\item[(5)] Let $\mathfrak{p}_1 = \bar{\mathfrak{p}}_1 \cap P$ and $\mathfrak{p}_2 = \bar{\mathfrak{p}}_2 \cap P$.

\item[(6)] We apply the algorithm recursively to $I+\langle 6\rangle$.

\item[(11)] Here we calculate the prime factorization $6 = 2 \cdot 3$.

\item[(14)] We determine the minimal associated primes $\bar{\mathfrak{p}}_3 = \langle x+1,y,z \rangle$ and
$\bar{\mathfrak{p}}_4 = \langle x,y,z \rangle$ of $I \mathbb{F}_3[x,y,z]$ using Algorithm~\ref{alg:primary_decomp}.
    
\item[(15)] Their canonical liftings are $\mathfrak{p}_3 = \langle x+1,y,z ,3\rangle$ and 
$\mathfrak{p}_4 = \langle x,y,z ,3\rangle$.  

\item[(14)] We determine the minimal associated primes $\bar{\mathfrak{p}}_5 = \langle x+1,y,z \rangle$ and
$\bar{\mathfrak{p}}_6 = \langle x,y,z \rangle$ of $I \mathbb{F}_2[x,y,z]$ using Algorithm~\ref{alg:primary_decomp}.
    
\item[(15)] Their canonical liftings are $\mathfrak{p}_5 = \langle x+1,y,z ,2\rangle$ and 
$\mathfrak{p}_6 = \langle x,y,z ,2\rangle$.  

\item[(7)] We  compute $J = \mathfrak{p}_1 \cap \mathfrak{p}_2 = \langle z,y,x^2+x-6 \rangle$. 
    
\item[(8)] The ideal~$J$ is contained in $\mathfrak{p}_3$, $\mathfrak{p}_4$, $\mathfrak{p}_5$, and $\mathfrak{p}_6$.

\item[(9)]The minimal associated prime ideals of~$R$ are given 
by $\mathfrak{p}_1$ and $\mathfrak{p}_2$.
\end{enumerate}
\end{exa}

\bigskip\bigbreak
%
%

\section{Computing Primitive Idempotents}
\label{sec5}

In this section our goal is to compute the primitive idempotents of an explicitly given finite $\Z$-algebra $R$.
We describe a variant of the method presented in Section~4 of~\cite{KMW} and analyze its complexity.
We will use the fact that the idempotents modulo a nilpotent ideal can be lifted. The following algorithm is
based on Lemma~3.2.1 in~\cite{DK}.

\begin{algorithm}{\bf (Lifting Idempotents)}\label{alg:lifting}\\
Let $R$ be an explicitly given finite $\Z$-algebra, and let $\Rad(0) \subseteq R$ be its nilradical.
Let $e \in R$ be such that $e^2 \equiv e \mod \Rad(0)$. Consider the following instructions.
\begin{enumerate}
\item[(1)] Set $h = e$.

\item[(2)] Compute $f = h+r-2hr$ where $r = h^2-h$.

\item[(3)] Represent $f^2-f$ as a $\Z$-linear combination $f^2-f = c_0g_0 + \cdots + c_ng_n$ using the
            structure constants.

\item[(4)] If $(c_0, \dots, c_n) \in \Syz_{\Z}(\mathcal{G})$, return~$f$ and stop. Otherwise set $h = f$ and continue
with step~(2).
\end{enumerate}
This is an algorithm which computes an idempotent $f \in R$ such that $f \equiv e \mod \Rad(0)$.
Furthermore, if $e$ is a primitive idempotent modulo $\Rad(0)$, then $f$ is a primitive idempotent in~$R$.
\end{algorithm}

\begin{proof}
The algorithm terminates since $\Rad(0)$ is a nilpotent ideal.
To prove the correctness, we show that if $h$ is an idempotent modulo $\Rad(0)^{2^k}$, then $f$ is an idempotent
modulo $\Rad(0)^{2^{k+1}}$. By assumption, we have $h^2-h \in \Rad(0)^{2^k}$, and therefore
$h^2r-hr = (h^2-h)r = r^2 \in \Rad(0)^{2^{k+1}}$. Then we get
$$
f^2 \;\equiv\; h^2+2hr-4h^2r \;\equiv\; h+r+2hr-4hr \;\equiv\; f \mod \Rad(0)^{2^{k+1}}
$$
and $f \equiv h \mod \Rad(0)^{2^k}$.
Now assume that $e \mod \Rad(0)$ is a primitive idempotent and that $f = e'+e''$ can be written as the sum of two
orthogonal idempotents. Then we have $e' \in \Rad(0)$ or $e'' \in \Rad(0)$, since $e$ is primitive. But $\Rad(0)$
consists only of nilpotent elements. Therefore $e'$ or $e''$ has to be zero.
\end{proof}

Let us apply this algorithm to an example.

\begin{exa}
Consider the finite $\mathbb{Z}$-algebra $R$ generated by $\mathcal{G} = \{1, g_1, \dots, g_5\}$ with relation ideal
$\langle 6,\, 3g_1,\, 3g_2,\, 3g_3\rangle$. The non-trivial structure constants
are given by $g_4^2 = g_5$, $g_5^2 = g_4$ and $g_4g_5 = 1$. We have
$\Rad(0) = \langle 6, 2g_5-2, g_2+2, g_2-g_3 \rangle$, and the residue class 
of $e = g_4+g_5+1$ in $R/\Rad(0)$ is a
primitive idempotent. We apply Algorithm~\ref{alg:lifting} to lift $e$ to an idempotent of~$R$.
\begin{enumerate}
\item[(1)] We set $h = e$.

\item[(2)] We compute $r = h^2-h = 2g_4+2g_5+2$ and $f = h+r-2hr = -3g_4-3g_5-3$.

\item[(3)] Using the structure constants, we calculate $f^2-f = 30g_4+30g_5+30$.

\item[(4)] Since $(30,0,0,0,30,30) \in \Syz_{\Z}(\mathcal{G})$, we return the primitive idempotent~$f$.

\end{enumerate}
\end{exa}

The number of iterations in Algorithm~\ref{alg:lifting} necessary to lift the idempotents can be bounded as follows.

\begin{prop}\label{prop:bound}
Let $R$ be a finite $\Z$-algebra of rank $r$, and let $T$ be the torsion subgroup of~$R^+$.
\begin{enumerate}
\item[(a)] We have $\Rad(0)^m = \{0\}$ for $m = r+ \length_{\Z}(T)$.

\item[(b)] Let $p_1^{e_1}, \dots, p_s^{e_s}$ be the elementary divisors of $R$. Then Algorithm~\ref{alg:lifting} 
terminates after at most $\lceil \log_2(r+e_1+ \cdots + e_s) \rceil$ steps.

\end{enumerate}
\end{prop}

\begin{proof}
To prove (a), note that an element $f \in \Rad(0)$ yields a nilpotent $\Z$-linear endomorphism $\varphi$
of $R$ given by multiplication with $f$. One therefore obtains a chain
$$
\Ker(\varphi) \subsetneq \Ker(\varphi^2) \subsetneq \cdots \subsetneq R.
$$
Now we show that if $\rank(\Ker(\varphi^i)) > 0$, then $\rank(\Ker(\varphi^i)) < \rank(\Ker(\varphi^{i+1}))$.
Note that $\rank(\Ker(\varphi^i)) = \rank(\Ker(\varphi^{i+1}))$ if and only if $\Ker(\varphi^{i+1})/\Ker(\varphi^i)$
is a torsion module. Let $\Ker(\varphi^{i+1})/\Ker(\varphi^i)$ be a torsion module. We prove by induction that
this implies $\Ker(\varphi^{i+k+1})/\Ker(\varphi^{i+k})$ is a torsion module for all $k \in \mathbb{N}$. For $k=0$
the claim is true by assumption. Now assume that $\Ker(\varphi^{i+k})/\Ker(\varphi^{i+k-1})$ is a torsion module,
and let $x \in \Ker(\varphi^{i+k+1})$. Then we have $\varphi(x) \in \Ker(\varphi^{i+k})$, and there exists a
non-zero $c \in \Z$ with $c \varphi(x) \in \Ker(\varphi^{i+k-1})$. Hence we obtain $cx \in \Ker(\varphi^{i+k})$.

Thus we conclude that $\Ker(\varphi^r)$ has rank $r$, and therefore that $\varphi^r(R)$ is a submodule of~$T$. This
forces $\varphi^{r+ \length_{\Z}(T)} = 0$.

Part (b) follows immediately from~(a), since the length of the torsion is given by the number of elementary divisors
$p_i^{e_i}$ counted with multiplicity~$e_i$.
\end{proof}

In order to compute the primitive idempotents of $R=P/I$, we can now use Algorithm~\ref{alg:lifting} to lift the
idempotents of~$R/\Rad(0)$. For the task of finding the primitive idempotents of $R/\Rad(0)$, we consider the minimal
associated primes of~$I$. Let us recall the following remark from~\cite{KMW}.

\begin{rem}
Let~$T$ be a commutative, unitary, noetherian ring.
\begin{enumerate}
\item[(a)] Given an idempotent $e\in T$, the set $\mathcal{V}(1-e)$
is both open and closed in~$\Spec(T)$.

\item[(b)] If $U \subseteq \Spec(T)$ is a subset which is both open and closed,
there exists a unique idempotent $e\in T$ such that in $T_{\mathfrak{p}}/\mathfrak{p} T_{\mathfrak{p}}$ 
we have $\bar{e}=1$ for $\mathfrak{p}\in U$ and $\bar{e}=0$ otherwise.

\item[(c)] The correspondence given in~(a) and~(b) is 1-1. The primitive idempotents
correspond uniquely to the connected components of~$\Spec(T)$.
\end{enumerate}
\end{rem}

Therefore, in order to compute the primitive idempotents of~$R/\Rad(0)$, we need to calculate the connected
components of~$\Spec(R/\Rad(0))$. Since the ring $R/\Rad(0)$ might have infinitely many prime ideals, we use the following
approach to describe the connected components of~$\Spec(R/\Rad(0))$.

\begin{defi}
Let $R$ be a finite $\Z$-algebra, and let $\minPrimes(R)$ be the set of minimal associated prime ideals of $R$.
A maximal subset of $\minPrimes(R)$ such that all corresponding prime ideals are part of the same connected
component of $\Spec(R)$ is called a \textbf{connected component} of~$\minPrimes(R)$.
\end{defi}

Since $R$ is a finite $\Z$-algebra, the associated primes of $R$ are either of
height~$n$ and do not contain a non-zero integer, or they are of height $n+1$ and hence maximal ideals.
Now Algorithm~\ref{alg:connected} determines the connected components of~$\minPrimes(R)$.

\begin{algorithm}{\bf(Computing the Connected Components of $\minPrimes(R)$)}\label{alg:connected}\\
Let $R$ be an explicitly given finite $\Z$-algebra. Consider the following sequence of instructions.
\begin{enumerate}
    \item[(1)] Compute the set of minimal associated prime ideals of $R$. Let $\mathfrak{m}_1,\dots,\mathfrak{m}_\ell$
        be the minimal associated prime ideals of height~$n+1$, and let $\mathfrak{p}_1,\dots,\mathfrak{p}_m$ be the
        minimal associated primes ideals of height~$n$.

    \item[(2)] Let $M = \{ \{\mathfrak{p}_1\}, \dots, \{\mathfrak{p}_m\}\}$.

    \item[(3)] While there are sets $C,C' \in M$ such that there exist $\mathfrak{p}_i \in C$ and $\mathfrak{p}_j \in C'$
        with $\mathfrak{p}_i + \mathfrak{p}_j \ne \langle 1 \rangle$ replace~$C$ and~$C'$ in~$M$ by $C \cup C'$.

    \item[(4)] For every ideal $\mathfrak{m}_i$, append the set $\{ \mathfrak{m}_i \}$ to~$M$.

    \item[(5)] Return $M$.
\end{enumerate}
This is an algorithm which computes a set $M=\{C_1,\dots,C_\nu\}$ such that $C_1,\dots,C_\nu$ are the connected
components of~$\minPrimes(R)$. It is in ZPPIF.
\end{algorithm}

\begin{proof}
The following observations show the correctness of this algorithm. An associated prime ideal of height~$n+1$
is maximal and therefore forms its own connected component. Two prime ideals $\mathfrak{p}_i$ and $\mathfrak{p}_j$ of
height~$n$ belong to the same connected component if and only if there is a maximal ideal $\mathfrak{m}$ containing
both $\mathfrak{p}_i$ and $\mathfrak{p}_j$, which is equivalent to $\mathfrak{p}_i + \mathfrak{p}_j \ne \langle 1 \rangle$.

Let us now show that the algorithm is in ZPPIF. The associated primes of $R$ can be computed in ZPPIF using
Algorithm~\ref{alg:ass_primes}. Then only two types of computations remain. Namely, we need to decide whether the sum
of two primes is equal to $\langle 1 \rangle$ and whether one prime ideal is contained in another. Both of these tasks
can be achieved in polynomial time by Proposition~\ref{prop:elementary_computations}.
\end{proof}

A more general version of this algorithm which computes the connected components of a set of (non-minimal) associated
prime ideals is given in Section~4 of~\cite{KMW}.

\begin{exa}\label{example:connected}
Consider the finite $\mathbb{Z}$-algebra $R$ given by the explicit presentation $R= \mathbb{Z}[x,y]/I$ with
$I = \langle x^2 + 5x,\, xy,\, y^2 - y,\, 6y \rangle$. We follow the above algorithm and compute the
connected components of $\minPrimes(R)$.
\begin{enumerate}
\item[(1)] Algorithm~\ref{alg:ass_primes} yields the minimal associated primes
$\mathfrak{m}_1 = \langle \bar{x}, \bar{y}-1, 3 \rangle$ and
$\mathfrak{m}_2 = \langle \bar{x}, \bar{y}+1, 2 \rangle$ of height~3, as well as
and the minimal associated primes $\mathfrak{p}_1 = \langle \bar{x},\bar{y} \rangle$ and
$\mathfrak{p}_2 = \langle \bar{y}, \bar{x}+5 \rangle$ of height~2.

\item[(2)] We let $M = \{ \{\mathfrak{p}_1\}, \{\mathfrak{p}_2\}\}$.

\item[(3)] Since $\mathfrak{p}_1 + \mathfrak{p}_2 = \langle \bar{x},\bar{y},5 \rangle \neq \langle 1 \rangle$,
we replace $\{\mathfrak{p}_1\}$ and $\{\mathfrak{p}_2\}$ by $\{\mathfrak{p}_1 + \mathfrak{p}_2\}$ and obtain
$M = \{ \{\mathfrak{p}_1 + \mathfrak{p}_2\}\}$.

\item[(4)] We add $\{\mathfrak{m}_1\}$ and $\{\mathfrak{m}_2\}$ to $M$.

\item[(5)] Thus the connected components of $\minPrimes(R)$ are
$\{\{\mathfrak{m}_1\}, \{\mathfrak{m}_2\}, \{\mathfrak{p}_1, \mathfrak{p}_2\}\}$.

\end{enumerate}
\end{exa}

From the connected components of $\minPrimes(R)$ we can now derive the primitive idempotents of~$R$.

\begin{algorithm}{\bf (Computing the Primitive Idempotents)}\label{alg:comp_prim_idemp}\\
Let $R$ be an explicitly given finite $\Z$-algebra.
The following steps define an algorithm which computes the primitive idempotents of~$R$ in ZPPIF.
\begin{enumerate}
\item[(1)] Compute the connected components $C_1,\dots,C_\nu$ of~$\minPrimes(R)$ using Alg.~\ref{alg:connected}.

\item[(2)] Compute $J = \bigcap_{\mathfrak{p} \in \minPrimes(R)} \mathfrak{p}$.

\item[(3)] For $i=1,\dots,\nu$, compute $J_i = \bigcap_{\mathfrak{p} \in C_i} \mathfrak{p}$.

\item[(4)] Compute the preimages $q_1, \dots, q_\nu$ of $e_1, \dots, e_\nu$ under the canonical
$\Z$-linear map $R/J \rightarrow R/J_1 \times \cdots \times R/J_\nu$.

\item[(5)] Using Algorithm~\ref{alg:lifting}, lift the idempotents $q_1, \dots, q_\nu$ of $R/J$ to idempotents
of~$R$ and return them.

\end{enumerate}
\end{algorithm}

\begin{proof}
For a proof of the correctness of this algorithm, we again refer to Section~4 of~\cite{KMW}.
Let us analyze the complexity of this algorithm. Step~(1) can be performed in ZPPIF using Algorithm~\ref{alg:connected}.
The remaining steps can performed in polynomial time by Proposition~\ref{prop:elementary_computations},
Algorithm~\ref{prop:preimage}, and Algorithm~\ref{alg:lifting}. The number of iterations necessary to perform
Algorithm~\ref{alg:lifting} has a polynomial bound by Proposition~\ref{prop:bound}.
\end{proof}

\begin{exa}
Let us continue Example~\ref{example:connected} and compute the primitive idempotents of the finite
$\mathbb{Z}$-algebra $\mathbb{Z}[x,y]/I$ with $I =\left\langle x^2 + 5x, xy, y^2 - y, 6y \right\rangle$.
\begin{enumerate}
\item[(1)] Using Algorithm~\ref{alg:connected}, we already computed the connected components
    $\{\mathfrak{m}_1\}$, $\{\mathfrak{m}_2\}$, and $\{\mathfrak{p}_1, \mathfrak{p}_2\}\}$ of $\minPrimes(R)$ in
    Example~\ref{example:connected}.

\item[(2)] Using Proposition~\ref{prop:elementary_computations}.e, we calculate
    $J = \mathfrak{p}_1 \cap \mathfrak{p}_2 \cap \mathfrak{m}_1 \cap \mathfrak{m}_2 = I$.

\item[(3)] Using Proposition~\ref{prop:elementary_computations}.e, we compute
    $\mathfrak{p}_1 \cap \mathfrak{p}_2 = \left\langle \bar{y}, \bar{x}^2+5\bar{x} \right\rangle$.

\item[(4)] We apply Algorithm~\ref{prop:preimage} to the $\Z$-linear map
    $R/I \rightarrow R/(\mathfrak{p}_1 \cap \mathfrak{p}_2) \times  R/\mathfrak{m}_1 \times R/\mathfrak{m}_2$
    and obtain the preimages $\bar{y}+1$, $3 \bar{y}$, $-2\bar{y}$ of $e_1, e_2, e_3$.

\item[(5)] Since $I=J$, we do not need to lift the idempotents. Thus we return the primitive idempotents
    $\bar{y}+1$, $3 \bar{y}$, $-2\bar{y}$ of~$R$.

\end{enumerate}
\end{exa}

\bigbreak
%
%

\section{Explicit $\Z$-Algebra Presentations and Strong Gr\"obner Bases}
\label{sec6}

In the previous sections we made the assumption that a $\Z$-algebra is explicitly given, i.e., given
by $\Z$-module generators, their linear relations, and structure constants. This information can be encoded
in an ideal $I \subseteq P = \Z[x_1, \dots, x_n]$ such that $R=P/I$. More precisely, let
\begin{equation*}
\tag{$\ast$} I = \langle x_i x_j - \ssum_{k=0}^n c_{ijk} x_k, \; \ssum_{k=0}^n a_{\ell k} g_k
\mid i,j=1, \dots, n, \; \ell=1, \dots, m \rangle
\end{equation*}
be the ideal in $P = \Z[x_1, \dots, x_n]$ encoding the information of an explicitly given
$\Z$-algebra $R=P/I$ as in Remark~\ref{remark:input}.

In this section we show that this representation of~$R$ is polynomial time equivalent
to computing a strong Gr\"obner basis of~$I$. This notion is defined as follows.

\begin{defi}
Let $I$ be an ideal in $P=\Z[x_1,\dots,x_n]$, and let~$\sigma$ be a term ordering on~$\mathbb{T}^n$.
A set of polynomials $G=\{g_1,\dots,g_r\}$ in~$I$ is called a \textbf{strong $\sigma$-Gr\"obner basis} 
of~$I$ if, for every polynomial $f\in I\setminus \{0\}$,
there exists an index $i\in\{1,\dots,r\}$ such that $\LM(f)$ is a multiple of $\LM(g_i)$.
\end{defi}

In the first subsection we show that a presentation $R=P/I$ with~$I$ as above allows us to compute 
a strong Gr\"obner basis of $I$ in polynomial time. In the second subsection we prove that, conversely, 
if $R=P/I$ and we know a strong Gr\"obner basis of~$I$, 
then we can calculate a presentation as in Remark~\ref{remark:input} in polynomial time.

%
%

\subsection{Computing a Strong Gr\"obner Basis of an Explicitly Given $\Z$-Algebra}

Let us begin with the task of computing a strong Gr\"obner basis for an ideal~$I$ as above. 
More generally, consider the following situation. Let
$P = \Z[x_1,\dots,x_n]$, and let $I_1, \dots, I_s \subseteq P$ be ideals such that $P/I_j$ is a finite
$\Z$-algebra for $j=1,\dots,s$. Our goal is to compute a strong Gr\"obner basis of their intersection
$I_1\cap \cdots \cap I_s$. For 0-dimensional ideals in a polynomial ring over a field, an intersection like this
can be computed using the generalized Buchberger-M\"oller algorithm (see~\cite{AKR}). 
In the following we extend this algorithm to ideals in $\Z[x_1,\dots,x_n]$.

Before formulating this generalization, we need to address the task of representing the residue classes
in $P/I_j$ using suitable systems of generators.

\begin{rem}\label{rem:reprvector}
Let $I$ be an ideal in~$P$ such that $P/I$ is a finite $\Z$-algebra, and let $\pi:\, P \longrightarrow P/I$ be the canonical
epimorphism. We need to be able to express the image $\pi(f)$ of an element $f \in P$ as a linear combination of
some system of $\Z$-module generators of $P/I$.

Let $(t_1, \dots, t_\mu)$ be a tuple of terms such that $\overline{\mathcal{O}} =
(\overline{t}_1, \dots, \overline{t}_\mu)$ generates $P/I$ as a $\Z$-module. 
Then a tuple $(a_1, \dots, a_\mu) \in \Z^\mu$ such that $\pi(f) = a_i \overline{t}_1+ \cdots+ 
a_\mu \overline{t}_\mu$ is called a \textbf{representation vector} of~$f$ with respect to~$\mathcal{O}$.
Representation vectors are in general not unique, but can be calculated efficiently in several settings.
\begin{enumerate}
\item[(a)] If~$I$ is given as in $(\ast)$ and $f\in P$, we can replace products $x_i x_j$ in~$f$
repeatedly by linear combinations $\ssum_{k=0}^n c_{ijk} x_k$ until the resulting polynomial~$g$ is linear.
Then $\pi(f)=\pi(g)$ is a $\Z$-linear combination of the residue classes of the terms in $\{1,x_1,\dots,x_n\}$.

\item[(b)] If~$I$ is given by a strong Gr\"obner basis with respect to a term ordering~$\sigma$, we
can use $\mathcal{O}_\sigma(I) = \mathbb{T}^n \setminus L$, where $L = \{m \in \LM(I) \mid \LC(m) = 1\}$,
and represent $\pi(f)$ for an element $f\in P$ by the residue class of its normal form $\NF_{\sigma,I}(f)$ 
which is a $\Z$-linear combination of the terms in~$\mathcal{O}_\sigma(I)$.
\end{enumerate}
In either case, if we have an implementation of a function that represents $\pi(f)$ for every polynomial $f\in P$
in the form $\pi(f) = a_1 \overline{t}_1+ \cdots + a_\mu \overline{t}_\mu$ then we write
$\RV_{\mathcal{O}}(f) = (a_1,\dots,a_\mu)$ for the corresponding representation vector.
\end{rem}

Now we can formulate the generalized Buchberger-M\"oller algorithm for ideals in $P = \Z[x_1, \dots, x_n]$.

\begin{algorithm}{\bf (Intersecting Ideals in $\Z[x_1,\dots,x_n]$)}\label{alg:BMintersection}\\
For $i=1,\dots,s$, let $I_1, \dots, I_s$ be ideals in~$P$ such that $P/I_i$ is a finite $\Z$-algebra,
and let $\mathcal{O}_i = \{t_{i1}, \dots, t_{i\mu_i}\} \subseteq \mathbb{T}^n$ be a set 
of~$\mu_i$ terms such that their residue classes generate $P/I_i$ as a $\Z$-module. Furthermore,
we assume that a $\Z$-submodule~$U_i$ of~$\Z^{\mu_i}$ is given such that the $\Z$-linear map
$P/I_i \longrightarrow \Z^{\mu_i}/U_i$ defined by $\overline{t}_{ij} \mapsto \overline{e}_i$ is an isomorphism.
Finally, let~$\sigma$ be a degree compatible term ordering on~$\mathbb{T}^n$. Consider the following instructions.
\begin{enumerate}
\item[(1)] Start with empty lists $G= [\;]$, $\mathcal{O}=[\;]$, $M = [\;]$, and a list $L=[1]$.

\item[(2)] Let $N = \{n_1, \dots, n_k \} \subseteq \Z^\mu$ such that
$\Z^{\mu_1}/U_1 \oplus \cdots \oplus \Z^{\mu_s}/U_s \cong \Z^\mu/\langle N \rangle$ for some $\mu\ge 1$.

\item[(3)] If $L$ is empty, return the pair $[G, \mathcal{O}]$ and stop. Otherwise, choose the power product
$t = \min_\sigma(L)$ and remove it from~$L$.

\item[(4)] Compute the vector $v = \RV_{\mathcal{O}_1}(t)\oplus \cdots \oplus \RV_{\mathcal{O}_s}(t) \in \Z^\mu$.

\item[(5)] Let $m_1, \dots, m_\ell$ be the elements of $M$. Compute a $\Z$-basis $B$ in Hermite normal form of
the set of solutions of the homogeneous linear equation
$$
v x_0 - \ssum_{i=1}^\ell m_i x_i  - \ssum_{i=\ell+1}^{k+\ell} n_i x_i = 0
$$
in the indeterminates $x_0, \dots, x_{k+\ell}$.

\item[(6)] If it exists, let $(a_i) \in \Z^{k+\ell+1}$ be a basis element in~$B$ with $a_0 \ne 0$. Append the
polynomial $a_0 t - \sum_{i=1}^\ell a_i t_i$ to the list $G$, where $t_i$ is the $i$-th power
product in the list $\mathcal{O}$.

\item[(7)] If there exists no such solution or if the first component $a_0$ of every solution is different from~1,
append the vector $v$ to $M$ and the term $t$ to the list~$\mathcal{O}$. Add to~$L$ those
elements of $\{x_1t, \dots, x_nt\}$ which are neither multiples of an element of~$L$ nor of
an element of $\{\LM(g) \mid g \in G\}$.

\item[(8)] Continue with step (3).
\end{enumerate}
This is an algorithm which computes a pair $(G, \mathcal{O})$ such that $G$ is a reduced strong $\sigma$-Gr\"obner
basis of $I = \bigcap_{i=1}^s I_s$ and the residue classes of the elements in $\mathcal{O}$ generate the
$\Z$-module $P/I$.
\end{algorithm}

\begin{proof}
First we prove correctness using induction on the iterations of the algorithm. More precisely, 
we show that if the values of~$G$ and~$\mathcal{O}$ are correct at the start of an iteration then 
they are still correct at the end of the iteration. 

If $L$ is not empty then it contains a minimal element~$t$ with respect to~$\sigma$. So, at the start 
of each iteration we have that the list~$G$ contains polynomials that can be extended to a minimal 
strong Gr\"obner basis of the intersection and whose leading terms are $\sigma$-smaller than~$t$. 
Consider the case $t >_\sigma 1$. If $(a_i) \in \Z^{k+\ell+1}$ is the solution of the linear system in step~(5) 
then $a_0 v - \sum_{i=1}^\ell a_i m_i \in \langle N \rangle$, and hence
$f = a_0 t - \sum_{i=1}^\ell a_i t_i \in I_1 \cap \cdots \cap I_s$. Since the solution space is in Hermite normal
form, every other polynomial~$h$ in the intersection with $\LT(h) = t$ has to satisfy $a_0 t \mid \LM(h)$, 
and~$f$ cannot be reduced further using the elements in~$G$. This means that a reduced strong Gr\"obner basis 
of the intersection has to contain~$f$, and~$f$ is added to~$G$ in step~(6).

If there exists no solution with non-zero first component, or the first component of all solutions is different from~1, 
there is no element~$g$ in~$G$ such that $\LM(g) \mid t$. Hence the term~$t$ has to be added to~$\mathcal{O}$.

Finally, the list~$L$ is updated such that its $\sigma$-smallest element is always the $\sigma$-smallest term
greater than~$t$ and not divisible by the leading monomial of some element of~$G$.
Since $P/(I_1\cap \cdots \cap I_s)$ is a finite $\Z$-module, there exists for every $i\in \{1,\dots,n\}$
a number $\alpha_i\ge 1$ such that $x_i^{\alpha_i} \in \LT(I_1\cap \cdots \cap I_s)$. Hence only a finite
number of terms can be added to the list~$L$. This proves that the procedure terminates.
\end{proof}

Let us apply this algorithm to an example.

\begin{exa}
Let $\sigma =$ \texttt{DegRevLex}, and consider the ideals $I = \langle 2x-y, x^2, y^2, xy \rangle$ and
$J = \langle x^2,y^2,2 \rangle$ in $P=\Z[x,y]$. We have
$$
\mathcal{O}_I = \mathbb{T}^n \setminus \langle x^2, y^2, xy \rangle = \{1,y,x\} \quad\text{and}\quad
\mathcal{O}_J = \mathbb{T}^n \setminus \langle x^2, y^2 \rangle = \{1,y,x,xy\}.
$$
The following table shows how Algorithm~\ref{alg:BMintersection} can be applied to compute a strong
$\sigma$-Gr\"obner basis of $I \cap J$. The first five rows of this table correspond to the elements of~$N$. The
algorithm considers the terms $1,y,x,y^2,xy,x^2$ in this order. Rows 6--11 in the table correspond to the
representation vectors computed in step~(4) of each iteration.
%
$$
\begin{array}{r|rrr|rrrrl}
         & 1 & y & x & 1 & y & x & xy & \\
         \cline{2-8}
         & 0 & -1 & 2 & 0 & 0 & 0 & 0 & \\
         & 0 & 0 & 0 & 2 & 0 & 0 & 0 & \\
         & 0 & 0 & 0 & 0 & 2 & 0 & 0 & \\
         & 0 & 0 & 0 & 0 & 0 & 2 & 0 & \\
         & 0 & 0 & 0 & 0 & 0 & 0 & 2 & \\
         \cline{2-8}
        1& 1 & 0 & 0 & 1 & 0 & 0 & 0 & \rightarrow G = [\;]\\
        y& 0 & 1 & 0 & 0 & 1 & 0 & 0 & \rightarrow G = [\;] \\
        x& 0 & 0 & 1 & 0 & 0 & 1 & 0 & \rightarrow G = [4x-2y]\\
        y^2& 0 & 0 & 0 & 0 & 0 & 0 & 0  & \rightarrow G = [4x-2y,\; y^2]\\
        xy& 0 & 0 & 0 & 0 & 0 & 0 & 1  & \rightarrow G = [4x-2y, \; y^2, \; 2xy]\\
        x^2& 0 & 0 & 0 & 0 & 0 & 0 & 0  & \rightarrow G = [4x-2y, \; y^2, \; 2xy, \; x^2]\\
\end{array}
$$
For instance, let us examine the 5-th iteration, where the algorithm handles the term~$xy$. In step~(5) of this
iteration we solve the homogeneous linear system of equations given by rows 1--8 and row 10 of the table. This is
because in the 4-th iteration we did not add the representation vector to the set~$M$. The Hermite normal form of a
basis of the solution space is given by
$$
\begin{bmatrix}
        2 & 0 & 0 & 0 & 0 & 0 & 0 & 0 & -1 \\
        0 & 4 & -2 & 0 & -2 & 0 & 1 & -2 & 0
\end{bmatrix}.
$$
Hence we append the element~$2xy$ to~$G$. Altogether, we obtain the strong $\sigma$-Gr\"obner
basis $\{4x-2y, y^2, 2xy, x^2\}$ of $I \cap J$.
\end{exa}

If the $\Z$-algebras $P/I_i$ determined by the ideals $I_i$ are
given as in Remark~\ref{remark:input}, then it is not necessary to compute their strong
Gr\"obner bases separately, as the following corollary shows.

\begin{cor}{\bf (Computing a Strong Gr\"obner Basis)}\label{cor:CompStrongGB}\\
Suppose that a finite $\Z$-algebra $R$ is explicitly given as in Remark~\ref{remark:input}, and let~$I$ be 
the ideal in~$P$ such that $R = P/I$ and~$I$ is of the form described in~$(\ast)$.
Let~$\sigma$ be a degree compatible term ordering on~$\mathbb{T}^n$. Then Algorithm~\ref{alg:BMintersection}
computes a strong $\sigma$-Gr\"obner basis of~$I$ in polynomial time.
\end{cor}

\begin{proof}
As mentioned in Remark~\ref{rem:reprvector}.a, the representation vector in step~(4) 
of an element $f \in P$ can be obtained by simplifying every product of indeterminates using the structure constants. 
The linear equation in step~(6) can be solved in polynomial time by Remark~\ref{remark:computs_over_Z}. 
The claim then follows from the fact that the number of iterations is bounded by the number of generators of~$R$.
\end{proof}

%
%

\subsection{Computing an Explicit Representation}
\label{subsec:6.2}

Now let $I \subseteq P$ be an ideal such that $R=P/I$ is a finite $\Z$-algebra. 
Given a strong Gr\"obner basis of~$I$ with respect to some term ordering~$\sigma$, our goal is
to compute a representation of~$R$ as in Remark~\ref{remark:input}.
The first step is to find a suitable system of $\Z$-module generators of~$R$.

\begin{prop}{\bf (Macaulay's Basis Theorem for Finite $\Z$-Algebras)}\label{prop:macaulay}\\
Let $I \subseteq P$ be an ideal such that $P/I$ is a finite $\Z$-algebra, let~$\sigma$
be a term ordering on~$\mathbb{T}^n$, and let $L = \{m \in \LM(I) \mid \LC(m) = 1\}$ be the set 
of all monic leading monomials of~$I$. 
Then the residue classes of the terms in $\mathcal{O}_\sigma = \mathbb{T}^n \setminus L$ form a 
generating set of the $\Z$-module $P/I$.
\end{prop}

\begin{proof}
It suffices to show that the $\Z$-submodule $Q = \sum_{t \in \mathcal{O}_\sigma} \Z(t+I) =
\sum_{t \in \mathcal{O}_\sigma} \Z t + I$ of~$P$ is equal to~$P$. Suppose that $Q \subsetneq P$, and let
$f \in P \setminus Q$ be a polynomial with minimal leading term. Then $\LT(f) \in \mathcal{O}$ 
would imply $f - \LC(f)\LT(f) \in P \setminus Q$ which would contradict the minimality of~$f$.
Hence we have $\LT(f) \in L$. This means there exist a polynomial $g \in I$ and a term $t \in \mathbb{T}^n$ 
such that $\LC(g) = 1$ and $\LT(f) = t \LT(g)$. But then $f - \LC(f)tg$ has smaller 
leading term, again contradicting the minimality of~$f$.
\end{proof}

The set $\mathcal{O}_\sigma$ in this proposition can be determined from a strong Gr\"obner basis of~$I$. 
Having obtained a tuple of $\Z$-module generators of $P/I$, it remains to determine its relation module.

For every polynomial $f \in P$, the Division Algorithm with respect to a $\sigma$-Gr\"obner basis of~$I$ yields its
normal form $\NF_{\sigma, I}(f) = \sum_{i=1}^\mu a_i t_i$ with $a_i \in \Z$ and $t_i \in \mathbb{T}^n$.
The canonical epimorphism $\pi:\, P \longrightarrow P/I$ satisfies $\pi(f) = \sum_{i=1}^\mu a_i \overline{t}_i$. 
Moreover, for a given generating set of $P/I$, we have a canonical surjective map $\varphi:\; P/I \longrightarrow \Z^\mu$. 
Combining this map with~$\pi$, we obtain a $\Z$-linear surjective map $\RV_{\mathcal{O}_\sigma}:\, 
P \longrightarrow \Z^\mu$ which sends~$f$ to $(a_1, \dots, a_\mu)$. 

Given a strong Gr\"obner basis of the ideal~$I$, 
we can now compute the kernel of the map~$\varphi$ as follows.

\begin{algorithm}{\bf (Computing a Module Presentation)}\label{alg:module_presentation}\\
Let $I$ be an ideal in~$P$, let $\sigma$ be a term ordering on~$\mathbb{T}^n$, let~$G$ 
be a minimal strong $\sigma$-Gr\"obner basis of~$I$, and let $\{t_1, \dots, t_k\} = \mathbb{T}^n \setminus L$, 
where~$L$ equals $\{m \in \LM(I) \mid \LC(m) = 1\}$ as in Proposition~\ref{prop:macaulay}.
Consider the following instructions.
\begin{enumerate}
\item[(1)] Start with an empty list $U = [\;]$ and $\mathcal{O} = [t_1, \dots, t_k]$.

\item[(2)] If $\mathcal{O}$ is empty, return the list~$U$ and stop. Otherwise, choose a term~$t$ in~$\mathcal{O}$
and remove it from~$\mathcal{O}$.

\item[(3)] Find the smallest integer $\ell > 1$ such that $\ell\, s \in \LM(G)$ for some $s \in \mathbb{T}^n$ with $s \mid t$. 
If no such integer exists, continue with step~(2).

\item[(4)] Let $c \in \Z^k$ be the coefficient vector representing~$\ell\,t$,
and let $d\in\Z^k$ be the coefficient vector representing $\NF_{\sigma,I}(\ell\, t)$ with respect to
$(t_1, \dots, t_k)$. Append $c-d$ to~$U$ and continue with step~(2).
\end{enumerate}
This is an algorithm which computes a list of tuples $U\subseteq \Z^k$ such 
that we have $P/I \cong \Z^k/ \langle U \rangle$.
\end{algorithm}

\begin{proof}
By Proposition~\ref{prop:macaulay}, the residue classes of the terms $t_1, \dots, t_k$ generate the
$\Z$-module~$P/I$. Assume that $t_1 >_\sigma t_2 >_\sigma \cdots >_\sigma t_k$, and consider the
$\Z$-module homomorphism
$$
\varphi:\; \Z^k \longrightarrow P/I \quad\hbox{given by\ } (c_1, \dots, c_k) \mapsto c_1 t_1 + \cdots + c_k t_k.
$$
Clearly, we have $\Ker(\varphi) \supseteq \langle U \rangle$. Assume that the converse containment does not hold.
Then there exists a tuple $c = (c_1, \dots, c_k) \in \Ker(\varphi)$ such that  $f = c_1 t_1 + \cdots + c_k t_k \in I$ and
$c \notin U$. Choose~$c$ with this property such that $\LT(f)$ is minimal. Since $f \in I$, there exists a polynomial
$g \in G$ with $\LM(g) \mid \LM(f) = c_i t_i$. 

Notice that this implies $\LC(g) > 1$. Otherwise, the term~$t_i$ would be divisible by~$\LT(g)$, and this 
would imply $t_i \in L$, a contradiction.

Consequently, there is an element $d = (d_1, \dots, d_k) \in U$ with
$d_1 = \cdots = d_{i-1} = 0$ and $\ell d_i  = c_i$ for some $\ell \in \Z$. The tuple $c-\ell d$ corresponds
to a polynomial whose leading term is smaller than $\LT(f)$. This shows $c-\ell d \in U$, but then we get
$c \in U$ in contradiction to our assumption. Hence we have the equality $\Ker(\varphi) = \langle U \rangle$
and $\varphi$ induces the inverse of the desired isomorphism.
\end{proof}

Given a strong Gr\"obner basis of an ideal~$I$ as above, an explicit representation of~$P/I$ 
can now be obtained as follows.

\begin{cor}{\bf (Computing an Explicit Representation)}\\
Let $I$ be an ideal in~$P$ such that $P/I$ is a finite $\Z$-algebra, let~$\sigma$ be a term ordering
on~$\mathbb{T}^n$, and let~$G$ be a minimal strong $\sigma$-Gr\"obner basis of~$I$. Consider the following
instructions.
\begin{enumerate}
\item[(1)] Compute the set $\{t_1, \dots, t_k\} = \mathbb{T}^n \setminus L$, where the set~$L$ equals
$\{m \in \LM(I) \mid \LC(m) = 1\}$.

\item[(2)] Apply Algorithm~\ref{alg:module_presentation} to compute generators of a submodule
$V \subseteq \mathbb{Z}^k$ such that $P/I \cong \mathbb{Z}^k/V$.

\item[(3)] For $i,j = 1,\dots,k$, use the Division Algorithm with respect to~$G$ to compute the normal form
$\NF_{\sigma, I}(t_i t_j) = \sum_{\ell=1}^k c_{ij\ell} t_\ell$ with $c_{ij\ell}\in\Z$.

\item[(4)] Return the residue classes of $t_1, \dots, t_k$ in~$P/I$, the generators of~$V$, and
the coefficients $c_{ij\ell}$ for $i,j,\ell = 1,\dots,k$.
\end{enumerate}
This is an algorithm which computes an explicit representation of $P/I$ in polynomial time in the bit complexity
of the input~$G$.
\end{cor}

\begin{proof}
The residue classes of $t_1, \dots, t_k$ generate $P/I$ as a $\Z$-module by Proposition~\ref{prop:macaulay}.
Algorithm~\ref{alg:module_presentation} then correctly computes~$V$ such that $P/I \cong \mathbb{Z}^k/V$. 
Finally, we check that $\NF_{\sigma, I}(t_i t_j)$ is of the form given in step~(4). Let~$s$ be a term not
contained in $\{t_1, \dots, t_k\}$. Suppose that $\NF_{\sigma, I}(t_i t_j)$ contains a monomial $ds$ with
$d \in \mathbb{Z}$. Then we have $ds \notin \LM(I)$, and in particular $s \notin L$, a contradiction.
The time complexity of each step is clearly polynomial.
\end{proof}

In conclusion, we can see that an explicit representation of a finite $\Z$-algebra~$R$ as in
Remark~\ref{remark:input} is equivalent to knowing a strong Gr\"obner basis of an ideal~$I$
in $P=\Z[x_1,\dots,x_n]$ such that $R=P/I$. Traditionally, many of the algorithms presented in this 
paper were executed using the calculation of a strong Gr\"obner basis. However, as there is no polynomial
time bound for a suitable version of Buchberger's Algorithm, the complexity bounds shown here would be
impossible if~$R$ were only given via $R=P/I$. As explicit representations of the type described in
Remark~\ref{remark:input} occur in many contexts (see for instance~\cite{KMW}), we hope that the
algorithms and complexity bounds developed here may prove useful.

\bigbreak

\end{document}